\documentclass[12pt,a4paper,reqno]{amsart}

\usepackage[utf8]{inputenc}
\usepackage[british]{babel}
\usepackage{amsmath,amsthm,amssymb}
\usepackage{hyperref,geometry}
\usepackage[dvipsnames]{xcolor}
\usepackage{enumitem}
\usepackage[abbrev, msc-links, nobysame]{amsrefs}
\usepackage{thm-restate}
\usepackage[nameinlink]{cleveref}
\usepackage{tikz}

\def\N{\mathbb N}
\def\Z{\mathbb Z}

\def\F{\mathcal F}

\hypersetup{
	colorlinks,
	linkcolor={red!60!black},
	citecolor={green!60!black},
	urlcolor={blue!60!black},
}

\colorlet{DRayFinitePhaseColour}{blue}
\colorlet{DRayFirstCrossingColour}{OliveGreen}
\colorlet{DRaySecondCrossingColour}{Dandelion}
\colorlet{DRayThirdCrossingColour}{Mahogany}

\newcommand{\braces}[1]{\lbrace #1 \rbrace}

\newcommand{\E}{\mathcal{E}}
\newcommand{\tribe}{\mathcal{F}}
\newcommand{\colex}{\text{col}}

\newtheorem{theorem}{Theorem}[section]
\newtheorem{lemma}[theorem]{Lemma}

\theoremstyle{definition}

\newtheorem{problem}[theorem]{Problem}
\theoremstyle{remark}
\newtheorem{claim}{Claim}

\def\lqedsymbol{\ifmmode$\lrcorner$\else{\unskip\nobreak\hfil
		\penalty50\hskip1em\null\nobreak\hfil$\rule{1.2ex}{1.2ex}$
		\parfillskip=0pt\finalhyphendemerits=0\endgraf}\fi} 

\newenvironment{claimproof}[1][\proofname]
{%
	\proof[#1]%
}
{%
	\endproof%
}

\makeatletter
\@addtoreset{claim}{theorem}
\makeatother

\crefname{definition}{definition}{definitions}
\crefformat{definition}{#2Definition~#1#3}
\Crefformat{definition}{#2Definition~#1#3}

\crefname{section}{section}{sections}
\crefformat{section}{#2Section~#1#3}
\Crefformat{section}{#2Section~#1#3}

\crefname{subsection}{subsection}{subsections}
\AtBeginDocument{\crefformat{subsection}{#2Subsection~#1#3}}
\Crefformat{subsection}{#2Subsection~#1#3}

\crefname{lemma}{lemma}{lemmata}
\crefformat{lemma}{#2Lemma~#1#3}
\Crefformat{lemma}{#2Lemma~#1#3}

\crefname{remark}{remark}{remarks}
\crefformat{remark}{#2Remark~#1#3}
\Crefformat{remark}{#2Remark~#1#3}

\crefname{theorem}{theorem}{theorems}
\crefformat{theorem}{#2Theorem~#1#3}
\Crefformat{theorem}{#2Theorem~#1#3}

\crefname{corollary}{corollary}{corollaries}
\crefformat{corollary}{#2Corollary~#1#3}
\Crefformat{corollary}{#2Corollary~#1#3}

\crefname{figure}{figure}{figures}
\crefformat{figure}{#2Figure~#1#3}
\Crefformat{figure}{#2Figure~#1#3}

\crefname{proposition}{proposition}{propositions}
\crefformat{proposition}{#2Proposition~#1#3}
\Crefformat{proposition}{#2Proposition~#1#3}

\crefname{observation}{observation}{observations}
\crefformat{observation}{#2Observation~#1#3}
\Crefformat{observation}{#2Observation~#1#3}

\crefname{claim}{claim}{claims}
\crefformat{claim}{#2Claim~#1#3}
\Crefformat{claim}{#2Claim~#1#3}

\crefname{problem}{problem}{problem}
\crefformat{problem}{#2Problem~#1#3}
\Crefformat{problem}{#2Problem~#1#3}

\crefname{question}{question}{question}
\crefformat{question}{#2Question~#1#3}
\Crefformat{question}{#2Question~#1#3}

\title{On the ubiquity of oriented double rays}
\author{Florian Gut
	\and
	Thilo Krill
	\and
	Florian Reich}

\address{Universit\"at Hamburg, Department of Mathematics, Bundesstrasse 55 (Geomatikum), 20146 Hamburg, Germany}
\email{\{florian.gut, thilo.krill, florian.reich\}@uni-hamburg.de}

\begin{document}

\begin{abstract}
A digraph $H$ is called \emph{ubiquitous} if every digraph that contains arbitrarily many vertex-disjoint copies of $H$ also contains infinitely many vertex-disjoint copies of $H$.
We study oriented double rays, that is, digraphs $H$ whose underlying undirected graphs are double rays.
Calling a vertex of an oriented double ray a turn if it has in-degree or out-degree 2, we prove that an oriented double ray with at least one turn is ubiquitous if and only if it has a (finite) odd number of turns.
It remains an open problem to determine whether the consistently oriented double ray is ubiquitous.
\end{abstract}

\maketitle

\section{Introduction}
A (di)graph $H$ is called \emph{ubiquitous} if any (di)graph $G$ that contains $k$ disjoint copies of $H$ for every $k \in \N$ also contains infinitely many disjoint copies of $H$ as sub(di)graphs.
While it is evident that every finite (di)graph is ubiquitous, the question whether an infinite (di)graph is ubiquitous is surprisingly challenging.
The study of ubiquity was sparked by the result of Halin stating that the ray is ubiquitous~\cite{halin1965}.
Subsequently, Halin~\cite{halin1970} also demonstrated the ubiquity of the double ray (see also Bowler, Carmesin and Pott~\cite{bowler2013}).
Moreover, one can consider notions of ubiquity with respect to containment relations other than the subgraph relation.
For results on minor-ubiquity, see Andreae~\cites{Andreae2002, Andreae2013} or Bowler, Elbracht, Erde, Gollin, Heuer, Pitz and Teegen~\cites{BEEGHPT22, BEEGHPT18, BEEGHPT20}.

The present authors initiated the investigation of ubiquity in digraphs \cite{GKR22} by examining the class of digraphs corresponding to rays in Halin's landmark result \cite{halin1965}.
More precisely, digraphs whose underlying undirected graphs are rays were considered, which are called \emph{oriented rays}.
In this context, a \emph{turn} is a vertex of in-degree or out-degree $2$.
The main result of \cite{GKR22} is:
\begin{theorem}[\cite{GKR22}*{Theorem 1.1}]\label{theo:ray_main_theorem}
An oriented ray is ubiquitous if and only if it has a finite number of turns.
\end{theorem}
In this paper we extend the investigation of ubiquity in digraphs to the class of digraphs whose underlying undirected graphs are double rays.
We call these digraphs \emph{oriented double rays} and tackle the following problem posed in \cite{GKR22}:
\begin{problem}\label{prob:are_oriented_double_rays_ubiquitous}
Which oriented double rays are ubiquitous?
\end{problem}
It turns out that -- in contrast to oriented rays -- not only is it relevant whether the number of turns is finite or infinite, but also the parity of this number plays a role.
The main result of this paper reads as follows:

\begin{restatable}{theorem}{DoubleRayUbiquity}\label{theo:dray_main_theorem}
An oriented double ray with at least one turn is ubiquitous if and only if it has a (finite) odd number of turns.
\end{restatable}
	
\cref{theo:dray_main_theorem} solves \cref{prob:are_oriented_double_rays_ubiquitous}, except for the case of the oriented double ray without any turns.
This raises the following problem:

\begin{problem}[\cite{GKR22}*{Problem 1.3}]
Is the oriented double ray without turns ubiquitous?
\end{problem}
	
The proof of \Cref{theo:dray_main_theorem} is constructive, employing results from \cite{GKR22} and investigating symmetry properties of oriented double rays.
	
We begin by proving in \cref{sec:dray_negative_results} that $R$ is non-ubiquitous if it has an even but non-zero or infinite number of turns:
In \cref{subsec:dray_negative_even_number_of_turns}, we address the case where $R$ has an even but non-zero number of turns.
If $R$ has infinitely many turns, we distinguish whether $R$ is non-periodic (\cref{subsec:dray_negative_non-periodic}) or periodic (\cref{subsec:dray_negative_periodic}) (periodicity is defined in \cref{sec:dray_preliminaries}).
Finally, in \cref{sec:dray_positive_result} we show that any oriented double ray with an odd number of turns is ubiquitous.

\section{Preliminaries}\label{sec:dray_preliminaries}

For introductions to graph theory and digraph theory we refer to the books of Diestel~\cite{D17} and Bang-Jensen and Gutin~\cite{bang2008}, respectively.
Any definitions not introduced here can be found in these books.

For digraphs $D$ and $D'$, we write $D \cong D'$ if $D$ is isomorphic to $D'$ and $D \leq D'$ if $D$ is isomorphic to a subdigraph of $D'$.

\subsection{Rays and double rays}

In this paper, rays and double rays are digraphs together with linear orders on their sets of vertices: A \emph{ray} or \emph{double ray} is a digraph $R$ together with a linear order $\leq_R$ on $V(R)$ isomorphic to $\N$ or $\Z$, respectively, such that for all vertices $v, w$ of $R$:
\begin{itemize}
\item if $v$ and $w$ are consecutive in the linear order, then either $vw \in A(R)$ or $wv \in A(R)$ but not both, and
\item if $v$ and $w$ are not consecutive in the linear order, then $vw, wv \notin A(R)$.
\end{itemize}
Hence, rays and double rays are oriented counterparts of undirected rays and double rays, together with linear orders.

We say that a (double) ray is a \emph{subdigraph} of another (double) ray or is \emph{isomorphic} to another (double) ray if this is true for the digraphs in the usual sense, regardless of the linear orders.

Let $R$ be a ray or a double ray and let $v \in V(R)$.
We write $vR$ for the subdigraph of $R$ induced by all vertices $w$ of $R$ with $w \geq_R v$ and $Rv$ for the subdigraph of $R$ induced by all vertices $w$ of $R$ with $w \leq_R v$.
An infinite subdigraph of this form is called a \emph{tail} of $R$.
We call $R$ \emph{periodic} if $R$ has a non-trivial $\leq_R$-preserving endomorphism.
(Note that when $R$ is a ray, any endomorphism must preserve $\leq_R$.
However, the same is not true for double rays.)
Let $v \in V(R)$ and let $f$ be a non-trivial $\leq_R$-preserving endomorphism of $R$ such that the distance $d$ between $v$ and $f(v)$ in $R$ is minimal.
Then we say that $R$ has \emph{periodicity} $d$ (and $d$ is independent of the choices of $v$ and $f$).

We call the unique $\leq_R$-minimal vertex of a ray $R$ the \emph{root} of $R$. A vertex of a ray or double ray which is incident with two outgoing or two incoming arcs is called a \emph{turn}.
A maximal (possibly infinite) directed path contained in a ray or double ray, i.e.\ a maximal connected subdigraph whose arcs are consistently oriented, is called a \emph{phase}.
A ray without turns is an \emph{in-ray} if the first arc points towards the root and an \emph{out-ray} otherwise.

\subsection{Tribes}

Corresponding to \cite{BEEGHPT22}*{Definition 5.1} and \cite{GKR22}, given two digraphs $D$ and $R$, we call a collection $\tribe$ of finite sets of disjoint copies of $R$ in $D$ an \emph{$R$-tribe in $D$}.
If $D$ is clear from the context, we omit the host digraph $D$.
For an $R$-tribe $\tribe$, we call any element of $\bigcup \tribe$ a \emph{member of $\tribe$}.
We say that $\tribe$ is \emph{thick} if for each natural number $n$ there is $F \in \tribe$ with at least $n$ elements.
Note that if $D$ contains arbitrarily many disjoint copies of $R$ if and only if $D$ contains a thick $R$-tribe.

In this paper, whenever we consider a copy $R'$ of a digraph $R$ we implicitly fix an isomorphism $\varphi\colon R \to R'$ and for a subdigraph $\hat{R} \subseteq R$ we write in short $\hat{R}'$ for $\varphi(\hat{R})$.
With this, we say that an $R$-tribe $\tribe$ is \emph{forked at} $\hat{R}$ if $\hat{R}' \cap {R}'' = \emptyset$ for any two distinct members $R', R''$ of $\tribe$.

\section{Negative results}\label{sec:dray_negative_results}

In this \namecref{sec:dray_negative_results} we prove the backwards implication of \cref{theo:dray_main_theorem}, which is divided into three different parts.
First we show that double rays with a (finite) even, non-zero number of turns are non-ubiquitous in \cref{theo:dray_even_number_of_turns}.
Then we show that non-periodic double rays with infinitely many turns are non-ubiquitous in \cref{theo:dray_negative_non_periodic}, and lastly that periodic double rays with infinitely many turns are non-ubiquitous in \cref{theo:dray_negative_periodic}. 

\subsection{Double rays with an even, non-zero number of turns}\label{subsec:dray_negative_even_number_of_turns}
Any double ray $R$ with an even, non-zero number of turns contains an in-ray and an out-ray. By glueing together in- and out-rays of the members of a thick $R$-tribe in a specific way we can show:
\begin{theorem}\label{theo:dray_even_number_of_turns}
Any double ray with an even, non-zero number of turns is non-ubiquitous.
\end{theorem}

\begin{proof}
Let $R$ be a double ray with an even, non-zero number of turns. Let $s$ be the first and $t$ be the last turn of $R$.
Since the number of turns is even, exactly one of $R s$ and $t R$ is an in-ray and exactly one an out-ray.
By possibly reversing the order $\leq_R$, we may assume that the former is an in-ray and the latter an out-ray.
Let $p \in \N$ be the length of the longest finite phase of $R$.
We will construct a digraph $D$ containing arbitrarily many but not infinitely many disjoint copies of $R$.
	
For the construction of $D$, we use the auxiliary set
\[
I:=\{(n,m) \in \N^2 \colon n \leq m\}
\]
ordered by the colexicographic order $\leq_\colex$.
For $(n,m) \in I$, we write $(n,m)^+$ for the successor of $(n,m)$ under $\leq_\colex$ and, if there exists a predecessor of $(n,m)$ under $\leq_\colex$, we refer to this predecessor as $(n,m)^-$.

Let $(R(n,m))_{(n,m)\in I}$ be a family of pairwise disjoint copies of $R$.
For $(n, m) \in I$, write $s(n,m)$ for the first turn of $R(n,m)$ and $t(n,m)$ for the last turn of $R(n,m)$.

Next, we define two families of vertices of $R(n,m)$ for every $(n, m) \in I$:
Let $(v_{(n,m)}^{(i,j)})_{(i,j) \in I}$ be a family of vertices of $R(n,m) s(n,m)$ such that 
\begin{itemize}
\item the order $\leq_{R(n,m)}$ and the order induced by $\leq_\colex$ on the superindices are reversed on $\{v_{(n,m)}^{(i,j)}: (i,j) \in I\}$,
\item the vertices of $(v_{(n,m)}^{(i,j)})_{(i,j) \in I}$ have distance at least $p + 1$ to each other and to $s(n,m)$ in $R(n,m)$.
\end{itemize}
Let $(w_{(n,m)}^{(i,j)})_{(i,j) \in I}$ be a family of vertices of $t(n,m) R(n,m)$ such that 
\begin{itemize}
\item the order $\leq_{R(n,m)}$ and the order induced by $\leq_\colex$ on the superindices coincide on $\{w_{(n,m)}^{(i,j)}: (i,j) \in I\}$,
\item the vertices of $(w_{(n,m)}^{(i,j)})_{(i,j) \in I}$ have distance at least $p + 1$ to each other and to $t(n,m)$ in $R(n,m)$.
\end{itemize}

Let $D$ be the digraph constructed from the disjoint union
$\bigsqcup_{(n,m) \in I} R(n,m)$
by identifying the two vertices $v^{(k, \ell)}_{(i,j)}$ and $w^{(i, j)}_{(k, \ell)}$ for any $(i, j), (k, \ell) \in I$ with $j \neq \ell$ (see \cref{fig:dray_even_number_of_turns}).
We simply refer to the vertex of $D$ that is obtained by identification of $v^{(k, \ell)}_{(i,j)}$ and $w^{(i, j)}_{(k, \ell)}$ as $v^{(k, \ell)}_{(i,j)}$.

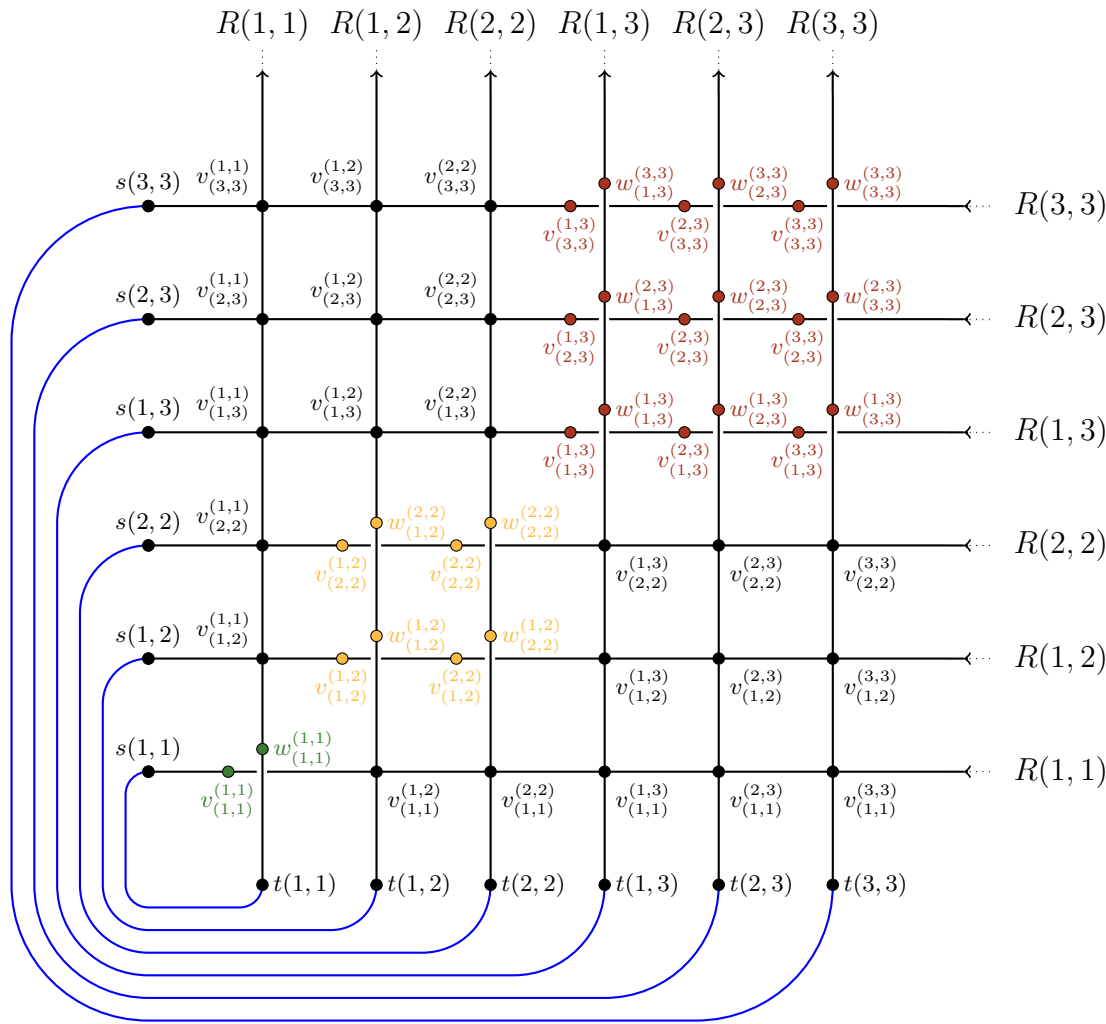
\begin{figure}[ht]
\centering
\begin{tikzpicture}[scale=1.5]

\foreach \x in {0,1,2,3,4,5,6,7} \foreach \y in {0,1,2,3,4,5,6,7} \node[shape=coordinate] (\x\y) at (\x,\y) {};

\foreach \x in {1,2,3,4,5,6} \draw[thick] (7\x) to (0\x);
\foreach \level/\start/\finish in {1/.95/1.04,2/1.95/2.04,2/2.95/3.04,3/1.95/2.04,3/2.95/3.04,4/3.95/4.04,4/4.95/5.04,4/5.95/6.04,5/3.95/4.04,5/4.95/5.04,5/5.95/6.04,6/3.95/4.04,6/4.95/5.04,6/5.95/6.04} \draw[ultra thick, white] (\start,\level) to (\finish,\level);
\foreach \y in {1,2,3,4,5,6} \draw[thick] (\y 0) to (\y 7);
\foreach \x in {1,2,3,4,5,6} \draw[thick,->] (\x,7) to (\x,7.2);
\foreach \x in {1,2,3,4,5,6} \draw[dotted] (\x,7.2) to (\x,7.4);
\foreach \y in {1,2,3,4,5,6} \draw[thick,->] (7.4,\y) to (7.15,\y);
\foreach \y in {1,2,3,4,5,6} \draw[thick] (7.15,\y) to (7,\y);
\foreach \y in {1,2,3,4,5,6} \draw[very thick,white] (7.187,\y) to (7.41,\y);
\foreach \y in {1,2,3,4,5,6} \draw[dotted] (7.2,\y) to (7.4,\y);
\draw[thick,DRayFinitePhaseColour] (01) to [out=180,in=90] (-.2,0.8) to (-.2,0) to [out=270,in=180] (0,-.2) to (0.8,-.2) to [out=0,in=270] (10);
\draw[thick,DRayFinitePhaseColour] (02) to [out=180,in=90] (-.4,1.6) to (-.4,0) to [out=270,in=180] (0,-.4) to (1.6,-.4) to [out=0,in=270] (20);
\draw[thick,DRayFinitePhaseColour] (03) to [out=180,in=90] (-.6,2.4) to (-.6,0) to [out=270,in=180] (0,-.6) to (2.4,-.6) to [out=0,in=270] (30);
\draw[thick,DRayFinitePhaseColour] (04) to [out=180,in=90] (-.8,3.2) to (-.8,0) to [out=270,in=180] (0,-.8) to (3.2,-.8) to [out=0,in=270] (40);
\draw[thick,DRayFinitePhaseColour] (05) to [out=180,in=90] (-1,4) to (-1,0) to [out=270,in=180] (0,-1) to (4,-1) to [out=0,in=270] (50);
\draw[thick,DRayFinitePhaseColour] (06) to [out=180,in=90] (-1.2,4.8) to (-1.2,0) to [out=270,in=180] (0,-1.2) to (4.8,-1.2) to [out=0,in=270] (60);

\foreach \x/\y in {.7/1,1/1.2} \draw[fill=DRayFirstCrossingColour] (\x,\y) circle [radius=.05];
\foreach \z in {2,3} \foreach \x/\y in {1.7/\z,2/\z.2,2.7/\z,3/\z.2} \draw[fill=DRaySecondCrossingColour] (\x,\y) circle [radius=.05];
\foreach \z in {4,5,6} \foreach \x/\y in {3.7/\z,4/\z.2,4.7/\z,5/\z.2,5.7/\z,6/\z.2} \draw[fill=DRayThirdCrossingColour] (\x,\y) circle [radius=.05];
\foreach \x in {1,2,3} \foreach \y in {4,5,6} \draw[fill] (\x\y) circle [radius=.05];
\foreach \x in {4,5,6} \foreach \y in {1,2,3} \draw[fill] (\x\y) circle [radius=.05];
\foreach \x in {1,2,3,4,5,6} \draw[fill] (\x 0) circle [radius=.05];
\foreach \y in {1,2,3,4,5,6} \draw[fill] (0\y) circle [radius=.05];
\foreach \x/\y in {1/2,1/3,2/1,3/1} \draw[fill] (\x\y) circle [radius=.05];
\foreach \y in {1,2,3,4,5,6} \draw[fill] (0\y) circle [radius=.05];

\foreach \x/\y/\z in {1/1/1,2/1/2,3/2/2,4/1/3,5/2/3,6/3/3} \node at (\x,7.6) {$R(\y,\z)$};
\foreach \x/\y/\z in {1/1/1,2/1/2,3/2/2,4/1/3,5/2/3,6/3/3} \node at (8,\x) {$R(\y,\z)$};
\foreach \x/\y/\z in {1/1/1,2/1/2,3/2/2,4/1/3,5/2/3,6/3/3} \node[right] at (\x 0) {\scriptsize{$t(\y,\z)$}};
\foreach \x/\y/\z in {1/1/1,2/1/2,3/2/2,4/1/3,5/2/3,6/3/3} \node[above] at (0\x) {\scriptsize{$s(\y,\z)$}};
\foreach \xpos/\ypos/\firsttopname/\secondtopname/\firstbottomname/\secondbottomname in {2/1/1/2/1/1,3/1/2/2/1/1,4/1/1/3/1/1,5/1/2/3/1/1,6/1/3/3/1/1,4/2/1/3/1/2,5/2/2/3/1/2,6/2/3/3/1/2,4/3/1/3/2/2,5/3/2/3/2/2,6/3/3/3/2/2} \node[below right] at (\xpos,\ypos) {\scriptsize{$v^{(\firsttopname,\secondtopname)}_{(\firstbottomname,\secondbottomname)}$}};
\foreach \xpos/\ypos/\firsttopname/\secondtopname/\firstbottomname/\secondbottomname in {1/2/1/1/1/2,1/3/1/1/2/2,1/4/1/1/1/3,1/5/1/1/2/3,1/6/1/1/3/3,2/4/1/2/1/3,2/5/1/2/2/3,2/6/1/2/3/3,3/4/2/2/1/3,3/5/2/2/2/3,3/6/2/2/3/3} \node[above left] at (\xpos,\ypos) {\scriptsize{$v^{(\firsttopname,\secondtopname)}_{(\firstbottomname,\secondbottomname)}$}};
\foreach \x/\y/\pos/\name in {.7/1/below/v,1/1.2/right/w} \node[\pos,DRayFirstCrossingColour] at (\x,\y) {\scriptsize{$\name^{(1,1)}_{(1,1)}$}};
\foreach \labelplace/\xpos/\ypos/\nodename/\firsttopname/\secondtopname/\firstbottomname/\secondbottomname in {below/1.7/2/v/1/2/1/2,right/2/2.2/w/1/2/1/2,below/1.7/3/v/1/2/2/2,right/2/3.2/w/2/2/1/2,below/2.7/2/v/2/2/1/2,right/3/2.2/w/1/2/2/2,below/2.7/3/v/2/2/2/2,right/3/3.2/w/2/2/2/2} \node[\labelplace,DRaySecondCrossingColour] at (\xpos,\ypos) {\scriptsize{$\nodename^{(\firsttopname,\secondtopname)}_{(\firstbottomname,\secondbottomname)}$}};
\foreach \labelplace/\xpos/\ypos/\nodename/\firsttopname/\secondtopname/\firstbottomname/\secondbottomname in {below/3.7/4/v/1/3/1/3,right/4/4.2/w/1/3/1/3,below/4.7/4/v/2/3/1/3,right/5/4.2/w/1/3/2/3,below/5.7/4/v/3/3/1/3,right/6/4.2/w/1/3/3/3,below/3.7/5/v/1/3/2/3,right/4/5.2/w/2/3/1/3,below/4.7/5/v/2/3/2/3,right/5/5.2/w/2/3/2/3,below/5.7/5/v/3/3/2/3,right/6/5.2/w/2/3/3/3,below/3.7/6/v/1/3/3/3,right/4/6.2/w/3/3/1/3,below/4.7/6/v/2/3/3/3,right/5/6.2/w/3/3/2/3,below/5.7/6/v/3/3/3/3,right/6/6.2/w/3/3/3/3} \node[\labelplace,DRayThirdCrossingColour] at (\xpos,\ypos) {\scriptsize{$\nodename^{(\firsttopname,\secondtopname)}_{(\firstbottomname,\secondbottomname)}$}};
\end{tikzpicture}

\caption{The digraph $D$ constructed in the proof of \cref{theo:dray_even_number_of_turns}. The horizontal paths represent parts of the in-rays $R(n,m)s(n,m)$ and the vertical paths represent parts of the out-rays $t(n,m)R(n,m)$. The blue paths connecting the vertices $s(n,m)$ and $t(n,m)$ represent the finite paths $s(n,m)R(n,m)t(n,m)$ that consist of the union of all finite phases of $R(n,m)$.}
\label{fig:dray_even_number_of_turns}
\end{figure}

The digraph $D$ contains $m$ disjoint copies of $R$ for every $m \in \N$, as the double rays $R(1, m), \dots, R(m,m)$ are disjoint.
Thus it remains to prove that $D$ does not contain infinitely many disjoint copies of $R$.
In \cref{clm:out-ray}, we investigate how an out-ray can lie in $D$, and in \cref{clm:in-ray}, we investigate how an in-ray with a fixed root can lie in $D$.
After that, we deduce from the two claims and from the fact that $R$ contains both an out-ray and an in-ray, that $D$ does not contain infinitely many copies of $R$.

\begin{claim}\label{clm:out-ray}
Any out-ray in $D$ has a tail that coincides with $x R(n,m)$ for some $(n, m) \in I$ and some $x \in V(R(n,m))$.
\end{claim}

\begin{claimproof}
Let $S$ be an arbitrary out-ray in $D$.
If $S$ is completely contained in some $R(n,m)$, then we are done.

Thus we can assume that $S$ is not completely contained in some $R(n,m)$.
This implies that $S$ must contain an identification vertex, i.e. $v^{(k,\ell)}_{(i,j)} \in V(S)$ for some $(i,j), (k, \ell) \in I$.
Let $(n,m) \in I$ be $\leq_\colex$-minimal with the property that $S$ contains a vertex $v^{(n,m)}_{(i,j)}$ for some $(i,j) \in I$.
We show that $v^{(n,m)}_{(i,j)} S$ is a tail of $t(n,m) R(n,m)$.

Suppose for a contradiction that $v^{(n,m)}_{(i,j)} S$ is not a tail of $t(n,m) R(n,m)$.
Set $v := v^{(n,m)}_{(i,j)}$ if $v^{(n,m)}_{(i,j)} \notin V(t(n,m) R(n,m))$ (i.e.\ if $j = m$), and otherwise let $v$ be the last vertex of $v^{(n,m)}_{(i,j)} S$ such that $v^{(n,m)}_{(i,j)} S v$ is contained in the out-ray $t(n,m) R(n,m)$.
In either case, there is $(k, \ell) \in I$ such that $v = v^{(n,m)}_{(k,\ell)}$ since in the latter case $v$ had to be identified with another vertex in the construction.
Then the first arc of $v S$ is contained in $R(k, \ell)$.
Since $S$ is an out-ray, it also contains $v^{(n,m)^-}_{(k,\ell)}$ if $(n,m) \neq (1,1)$, or $s(k,\ell)$ if $(n,m) = (1,1)$ (see \cref{fig:dray_even_number_of_turns}).
This contradicts either the minimality of $(n,m)$ or the fact that $s(k,\ell)$ is a turn, respectively.
\end{claimproof}

\begin{claim} \label{clm:in-ray}
For any $(n', m'), (n, m) \in I$, any in-ray in $D$ with root $s(n',m')$ contains a vertex of $\{ v^{(n,m)}_{(i, j)}: (i, j) \in I \}$.
\end{claim}

\begin{claimproof}
Let $(n,m), (n', m') \in I$ and let $S$ be an arbitrary in-ray in $D$ with root $s(n', m')$.
As no vertex of $s(n', m') S(n', m') t(n',m')$ has been identified with other vertices in the construction of $D$ and $t(n', m')$ is a turn, $S$ contains the vertex $v^{(1,1)}_{(n', m')}$.

We consider the set $\mathcal{X}:= \{ v^{(k, \ell)}_{(i, j)} \in V(S): (i,j), (k,\ell) \in I , (k, \ell) \leq_\colex (n,m) \}$ and let $(k, \ell) \in I$ be $\leq_\colex$-maximal with the property that an element of $\mathcal{X}$ has superindex $(k,\ell)$.
We show $(k,\ell) = (n, m)$, which implies the statement of \cref{clm:in-ray}.

Suppose for a contradiction that $(k,\ell) <_\colex (n, m)$.
Let $(i, j) \in I$ be $\leq_\colex$-minimal with the property that $v^{(k,\ell)}_{(i, j)}$ is an element of $\mathcal{X}$.
The first arc of $v^{(k,\ell)}_{(i, j)} S$ lies either in $R(i,j)s(i,j)$ or in $t(k,\ell)R(k,\ell)$.
In the former case, the in-ray $S$ must also contain $v^{(k,\ell)^+}_{(i, j)}$ (see \cref{fig:dray_even_number_of_turns}), contradicting the maximality of $(k,\ell)$.
In the latter case, $S$ also contains $t(k,\ell)$ if $v^{(k,\ell)}_{(i, j)}$ is the first vertex of $t(k, \ell) R(k, \ell)$ that has been identified, or $S$ contains a vertex $v^{(k,\ell)}_{(i', j')}$ for $(i',j') <_\colex (i,j)$ otherwise (see \cref{fig:dray_even_number_of_turns}).
This contradicts either the fact that $t(k,\ell)$ is a turn or the minimality of $(i,j)$, respectively.
\end{claimproof}
		
Let $\hat R$ be an arbitrary copy of $R$ in $D$.
By \cref{clm:out-ray}, there is $(n,m) \in I$ and $x \in V(R(n,m))$ such that an out-ray in $\hat R$ coincides with $x R(n,m)$.
To prove that $D$ cannot contain infinitely many disjoint copies of $R$, it suffices to show that any copy of $R$ in $D$ has an in-ray that starts in some $s(n',m')$. Then by \Cref{clm:in-ray}, every copy of $R$ in $D$ contains a vertex of $\{v^{(n,m)}_{(i,j)}: (i,j) \in I\}$, contradicting that the set $\{v^{(n,m)}_{(i,j)}: (i,j) \in I\} \setminus V(x R(n,m))$ is finite.

Let $\tilde{R}$ be any copy of $R$ in $D$ and let $\tilde{R} y$ be the unique phase of $\tilde{R}$ that forms an in-ray.
By construction of $D$, any phase of length at most $p$ of $\tilde{R}$ is contained in $s(i,j) R(i,j) t(i,j)$ for some $(i,j) \in I$, which immediately yields $y \in \braces{s(n',m'), t(n',m')}$ for some $(n',m')\in I$.
Since $\tilde{R} y$ is an in-ray, we must have $y = s(n', m')$ as desired. This completes the proof.
\end{proof}

\subsection{Non-periodic double rays with infinitely many turns}\label{subsec:dray_negative_non-periodic}
In \cref{lem:non-periodic_subray}, \cref{lem:periodic_dray,thm:periodic_dray} we investigate symmetry properties of non-periodic double rays and show that any such double ray $R$ has a tail which is isomorphic to only very specific other tails of $R$.
In \cref{theo:dray_negative_non_periodic}, we use this result to reduce the non-ubiquity of non-periodic double rays with infinitely many turns to the non-ubiquity of rays with infinitely many turns (\cref{theo:ray_main_theorem}).

\begin{lemma}\label{lem:non-periodic_subray}
For every non-periodic double ray $R$ there exists $v^* \in V(R)$ such that $R v^*$ is non-periodic.
\end{lemma}
	
\begin{proof}
Suppose for a contradiction that $R w$ is periodic for all $w \in V(R)$.
Let $p \in \N$ be minimal among the periodicities of $Rw$ for all $w \in V(R)$ and let $v \in V(R)$ such that $Rv$ has periodicity $p$.
We show that $Rw$ has periodicity $p$ for any $w \in V(R)$.
This then implies that $R$ is periodic with periodicity $p$, a contradiction.

For $w >_R v$, let $f$ be any non-trivial endomorphism of $Rw$, which exists since $Rw$ is periodic.
By concatenating $f$ with itself multiple times if necessary, we obtain an endomorphism of $Rw$ whose image is contained in $Rv$.
Thus it remains to prove that $Ru$ has periodicity $p$ for any $u <_R v$.
Since $Ru$ is a tail of $Rv$, $Ru$ has periodicity at most $p$, and by minimality of $p$, $Ru$ has periodicity exactly $p$.
\end{proof}

\begin{lemma} \label{lem:periodic_dray}
For every non-periodic double ray $R$ there exists $v^* \in V(R)$ such that for all $v \geq_R v^*$ and all $w \in V(R): $
\[
Rv \cong Rw \Rightarrow v = w .
\]
\end{lemma}
\begin{proof}
We choose $v^\ast$ such that $Rv^\ast$ is non-periodic according to \cref{lem:non-periodic_subray}. Let $v \geq_R v^*$ and $w \in V(R)$ such that $Rv \cong Rw$.
If $w <_R v$, we can restrict the isomorphism $Rv \to Rw$ to a non-trivial endomorphism of $R v^*$, contradicting that $R v^*$ is non-periodic.
If $w >_R v$, we can restrict the isomorphism $Rw \to Rv$ to a non-trivial endomorphism of $R v^*$, which again is a contradiction.
Thus $v = w$ holds.
\end{proof}

\begin{theorem} \label{thm:periodic_dray}
For every non-periodic double ray $R$ there exists $\hat{v} \in V(R)$ such that for all $w \in V(R)$:
\begin{enumerate}[label=(\arabic*)]
\item\label{item:periodic_dray1} $R \hat{v} \not\leq \hat{v} R, $
\item\label{item:periodic_dray2} $R \hat{v} \cong Rw \Rightarrow \hat{v}R \cong wR$, and
\item\label{item:periodic_dray3} $R \hat{v} \cong wR \Rightarrow \hat{v}R \cong Rw$.
\end{enumerate}
\end{theorem}
\begin{proof}
Let $v^*$ be as in \cref{lem:periodic_dray}.
\begin{enumerate}[label={\textbf{Case \arabic*:}}, ref=Case \arabic*, align=left, leftmargin=5pt]
\item \textbf{There is $v \geq_R v^*$ such that $R v \not\cong w R$ for all $w \in V(R)$.}

\noindent In this case we set $\hat{v}:= v$, which directly implies \labelcref{item:periodic_dray1} and \labelcref{item:periodic_dray3}.
For \labelcref{item:periodic_dray2}, since $ \hat{v} \geq v^*$ and $v^\ast$ was picked with the property of \cref{lem:periodic_dray}, $R \hat{v}  \cong R w $ implies $ \hat{v} = w$ and thus $\hat{v}R = wR$.

\item \textbf{For all $v \geq_R v^*$ there is $\alpha(v) \in V(R)$ such that $R v \cong \alpha(v) R$.}

\begin{claim}\label{clm:alpha_of_v_is_unique_for_v}
For every $v \geq_R v^*$, the vertex $\alpha(v)$ is unique.
\end{claim} 
\begin{claimproof}
Suppose for a contradiction that there are $\alpha(v) <_R \alpha(v)' \in V(R)$ such that $\alpha(v) R \cong R v \cong \alpha(v)' R$. Since $\alpha(v)' R$ is a proper tail of $\alpha(v) R$, it follows that $R v \cong \alpha(v)' R$ is isomorphic to a proper tail of $R v \cong \alpha(v) R$.
Thus there exists a non-trivial endomorphism of $Rv$, which contradicts that $v^*$ has the property of \cref{lem:periodic_dray}.
\end{claimproof}

\begin{claim}\label{clm:alpha_of_v_smaller_than_v}
There is $\hat{v} \geq_R v^*$ such that $\alpha(\hat{v}) <_R \hat{v}$.
\end{claim}
\begin{claimproof}
Let $v$ and $v'$ be any vertices of $R$ with $v^\ast \leq_R v <_R v'$.
Since $ R v$ is a proper tail of $ R v'$, it follows that $\alpha(v) R \cong R v$ is isomorphic to a proper tail of $\alpha(v') R \cong R v'$. Thus there is a vertex $w >_R \alpha(v')$ such that $w R \cong \alpha(v) R$.
Then $w = \alpha(v)$ by \cref{clm:alpha_of_v_is_unique_for_v}.
In conclusion, we have established that $\alpha(v) >_R \alpha(v')$ whenever $v <_R v'$.
Therefore, it is possible to pick a vertex $\hat{v}$ which is sufficiently large with respect to $\leq_R$, so that $\alpha(\hat{v}) <_R \hat{v}$.
\end{claimproof}
\noindent We show that any vertex $\hat{v}$ as in \cref{clm:alpha_of_v_smaller_than_v} satisfies properties \labelcref{item:periodic_dray1,item:periodic_dray2,item:periodic_dray3}.
For \labelcref{item:periodic_dray1}, assume that $R\hat{v} \leq \hat{v}R$.
This means that there is $v' \geq_R \hat{v}$ with $R\hat{v} \cong v'R$.
A contradiction since by \cref{clm:alpha_of_v_is_unique_for_v} $\alpha(\hat{v})$ is unique, but $\alpha(\hat{v}) <_R \hat{v} \leq_R v'$ by choice of $\hat{v}$.
For \labelcref{item:periodic_dray2}, recall that $\hat{v} \geq_R v^\ast$. Thus, \cref{lem:periodic_dray} implies that the only vertex $w \in V(R)$ with $R\hat{v} \cong Rw$ is $\hat{v}$.
This proves \labelcref{item:periodic_dray2} since clearly $\hat{v} R \cong \hat{v} R$.
For \labelcref{item:periodic_dray3}, as $\alpha(\hat{v})$ is unique by \cref{clm:alpha_of_v_is_unique_for_v}, it is enough to prove $\hat{v} R \cong R \alpha(\hat{v})$.
Let $\psi: R \hat v \to \alpha(\hat v) R$ be an isomorphism, which exists by choice of $\alpha(\hat v)$.
Then $\psi$ maps $\alpha(\hat v)$ to $\hat v$, since the distance between $\alpha(\hat v) \in V(R \hat v)$ and the root of $R \hat v$ and the distance between $\hat v \in V(\alpha(\hat v) R)$ and the root of $\alpha(\hat v) R$ are the same.
Thus $\psi$ can be restricted to an isomorphism $R \alpha(\hat v) \to \hat v R$.
\end{enumerate}
This completes the proof.
\end{proof}
Now we combine \cref{theo:ray_main_theorem} and \cref{thm:periodic_dray} to prove:
\begin{theorem}\label{theo:dray_negative_non_periodic}
Any non-periodic double ray $R$ with infinitely many turns is non-ubiquitous.
\end{theorem}

\begin{proof}
Let $R$ be any such double ray. We construct a digraph $D$ that contains arbitrarily many but not infinitely many copies of $R$.
Without loss of generality, for every $v \in V(R)$ the ray $vR$ contains infinitely many turns (otherwise reverse the order $\leq_R$).
Let $\hat{v} \in V(R)$ be as in \cref{thm:periodic_dray}.
As $\hat{v} R$ contains infinitely many turns, there is a digraph $D'$ containing arbitrarily many but not infinitely many disjoint copies of $\hat{v} R$ by \cref{theo:ray_main_theorem}.
We construct $D$ from $D'$ and a family $(S_x)_{x \in V(D')}$ of disjoint copies of $R \hat{v}$ by identifying the root of $S_x$ with $x$ for each $x \in V(D')$.

By construction, $D$ contains arbitrarily many disjoint copies of $R$.
We have to show that $D$ does not contain infinitely many disjoint copies of $R$, which implies the theorem.
It suffices to prove that each copy of $R$ in $D$ has a tail isomorphic to $\hat{v} R$ that is contained in the subdigraph $D'$ of $D$.
Then $D$ cannot contain infinitely many disjoint copies of $R$ since $D'$ does not contain infinitely many disjoint copies of $\hat{v} R$.

Let $\tilde{R}$ be any copy of $R$ in $D$.
If $\tilde{R}$ is completely contained in $D'$, we are done.
Thus we can suppose that there is $x \in V(D')$ and $w \in V(\tilde{R})$ such that either $S_x = \tilde{R} w$ or $S_x = w \tilde{R}$.

In the former case, we have $R \hat{v} \cong S_x = \tilde{R} w$ and thus $\hat{v} R \cong w \tilde{R}$ by \cref{thm:periodic_dray} \labelcref{item:periodic_dray2}.
It follows from \labelcref{item:periodic_dray1} that $R \hat{v} \not\le \hat{v} R \cong w \tilde{R}$.
Hence $w \tilde{R}$ cannot have a tail in any $S_y$ for $y \in V(D')$.
Thus $w \tilde{R}$ is the desired tail of $\tilde{R}$ which is isomorphic to $\hat{v} R$ and contained in $D'$.

Similarly, in the latter case, we have $R \hat{v} \cong S_x = w \tilde{R}$ and thus $\hat{v} R \cong \tilde{R} w$ by \labelcref{item:periodic_dray3}.
It follows from \labelcref{item:periodic_dray1} that $R \hat{v} \not\le \hat{v} R \cong \tilde{R} w$.
Hence $\tilde{R} w$ cannot have a tail in any $S_y$ for $y \in V(D')$ and $\tilde{R} w$ is the desired tail of $\tilde{R}$.
\end{proof}

\subsection{Periodic double rays with infinitely many turns}\label{subsec:dray_negative_periodic}
Let $R$ be a periodic double ray with infinitely many turns and let $\hat R, \tilde{R}$ be disjoint copies of $R$.
By periodicity of $R$, one can show that identifying a turn of $\hat R$ of out-degree $2$ and a turn of $\tilde{R}$ of in-degree $2$ results in a digraph in which a copy of $R$ has to be completely contained in either $\hat R$ or $\tilde{R}$.
We use this fact to prove:

\begin{theorem}\label{theo:dray_negative_periodic}
Any periodic double ray with infinitely many turns is non-ubiquitous.
\end{theorem}

\begin{proof}
Let $R$ be any periodic double ray with infinitely many turns and denote the periodicity of $R$ by $p \in \N$.
We will construct a digraph $D$ containing arbitrarily many but not infinitely many copies of $R$.

We set
\[I:=\{(n,m) \in \N^2 \colon n \leq m\}\]
and let $(R(n,m))_{(n,m)\in I}$ be a family of pairwise disjoint copies of $R$.
Let $D$ be the digraph constructed from the disjoint union $\bigsqcup_{(n,m) \in I} R(n,m)$ by identifying pairwise disjoint pairs of vertices such that for any $(n,m), (n',m') \in I$:

\begin{enumerate}[label=(\roman*)]
\item\label{item:m_eq_m'} no vertices of $R(n,m)$ and $R(n',m')$ have been identified with each other if $m = m'$,
\item\label{item:m_neq_m'} exactly one vertex of $R(n,m)$ and exactly one vertex of $R(n',m')$ have been identified if $m \neq m'$,
\item\label{item:different_out_degrees} if $v \in R(n,m)$ and $w \in R(n',m')$ have been identified with each other, then either the out-degree of $v$ in $R(n,m)$ is 2 and the out-degree of $v$ in $R(n',m')$ is 0 or vice versa, and
\item\label{item:distance_p+1} two vertices $v \neq w \in R(n,m)$ that have been identified with other vertices have distance at least $p$ in $R(n,m)$.
\end{enumerate}
A graph $D$ satisfying \labelcref{item:m_eq_m',item:m_neq_m',item:distance_p+1,item:different_out_degrees} can be constructed by enumerating all unordered pairs $\{R(n,m),R(n',m')\}$ of double rays with $m \neq m'$ and recursively identifying suitable turns of the two rays in each pair.

The digraph $D$ contains arbitrarily many disjoint copies of $R$, as the double rays $R(1, m), \dots, R(m,m)$ are disjoint for any $m \in \N$ by \ref{item:m_eq_m'}.
To prove that $D$ does not contain infinitely many disjoint copies of $R$, it suffices to show that any copy of $R$ in $D$ is of the form $R(n,m)$ for some $(n,m) \in I$:
Then any infinite family of disjoint copies of $R$ in $D$ would contain two rays $R(n,m)$, $R(n',m')$ with $m \neq m'$ by definition of $I$.
However, $R(n,m)$ and $R(n',m')$ are not disjoint in $D$ by \ref{item:m_neq_m'}.

Suppose for a contradiction that there is a copy $\hat R$ of $R$ in $D$ that is not contained in some $R(n,m)$ for $(n,m) \in I$.
Then there are $(n,m) \neq (n',m') \in I$ and $v \in V(\hat R)$ such that one arc of $\hat R$ incident with $v$ is contained in $R(n,m)$ and the other arc of $\hat R$ incident with $v$ is contained in $R(n',m')$.
We assume without loss of generality that the first arc of $v \hat R$ is contained in $v R(n,m)$ (and not in $R(n,m)v$, $R(n',m')v$ or $vR(n',m')$).

Since $R(n,m)$ has periodicity $p$, the arc $a$ of $ R(n,m)$ preceding $v$ has the same orientation as the $p$-th arc $a'$ of $R(n,m)$ succeeding $a$.
Similarly, the arc $b$ of $\hat R$ preceding $v$ has the same orientation as the $p$-th arc $b'$ of $\hat R$ succeeding $b$.
As no vertices of distance at most $p - 1$ to $v$ in $D$ other than $v$ were identified by \ref{item:distance_p+1}, the first $p$ arcs of $v \hat R$ coincide with the first $p$ arcs of $v R(n,m)$ and in particular we have $a' = b'$.
Hence the arcs $a, b$ either both point towards $v$ or both point away from $v$, contradicting \ref{item:different_out_degrees} since $a\in A(R(n',m'))$ and $b \in A(R(n,m))$.
\end{proof}

\section{Positive results}\label{sec:dray_positive_result}

In this \namecref{sec:dray_positive_result} we prove \cref{theo:double_rays_positive}, for which we need three results from \cite{GKR22}.
Firstly, we need the following lemma; its proof is straight-forward:

\begin{lemma}[\cite{GKR22}*{Lemma 2.1}]\label{lem:change_orientation}
A digraph $H$ is ubiquitous if and only if the digraph obtained by reversing the orientation of all arcs of $H$ is ubiquitous.
\end{lemma}

Secondly, we need \cite{GKR22}*{Lemma 3.3}, which enables us to thin out tribes to make them forked at given finite subdigraphs:

\begin{lemma}[\cite{GKR22}*{Lemma 3.3}]\label{lem:disjoint_initial_segments}
Let $D$ and $H$ be digraphs and let $\hat{H} \subseteq H $ be a finite subdigraph.
If there exists a thick $H$-tribe in $D$, then there is a thick $H$-tribe $\F$ in $D$ that is forked at $\hat{H}$.
\end{lemma}

Thirdly, we need the following result, which essentially states that pairs of disjoint out-rays are ubiquitous with an additional property for the roots of the out-rays.
In preparation we define:
Let $U$ be the disjoint union of two out-rays and $X$ a set of pairs of vertices.
We say that $U$ \emph{is rooted in} $X$ if there is $(x,y) \in X$, such that $x$ and $y$ are the roots of the two rays which are the components of $U$.
If $U = \{(x,y)\}$, we also say that $U$ \emph{is rooted in} $(x,y)$.

\begin{lemma}\label{theo:ubiquity_with_starting_set2}
Let $U$ be the disjoint union of two out-rays, let $D$ be a digraph and let $X$ be a set of pairs of vertices of $D$.
If there exists a thick $U$-tribe $\tribe$ in $D$ where all members of $\tribe$ are rooted in $X$, then $D$ contains infinitely many disjoint copies of $U$ that are rooted in $X$.
\end{lemma}

We omit the proof of \cref{theo:ubiquity_with_starting_set2} since it is very similar to the proof of \cite{GKR22}*{Theorem 3.2}.

\begin{theorem}\label{theo:double_rays_positive}
Any double ray with an odd number of turns is ubiquitous.
\end{theorem}

\begin{proof}
Let $R$ be a double ray with a (finite) odd number of turns, $D$ a digraph and $\E$ a thick $R$-tribe in $D$.
We show that $D$ contains infinitely many disjoint copies of $R$.
Let $\hat{R}$ be a finite connected subdigraph of $R$ that contains all turns of $R$ as internal vertices.
Since $R$ has an odd number of turns, by deleting the internal vertices of $\hat R$ the digraph $R$ falls apart into a disjoint union $U$ of either two in-rays or two out-rays.
By \cref{lem:change_orientation}, we may assume that the latter is the case.

Next, we apply \cref{lem:disjoint_initial_segments} to $D$, $R$, $\hat{R}$ and $\E$, which yields a thick $R$-tribe $\tribe$ in $D$ that is forked at $\hat{R}$.
Let $\tribe'$ be the $U$-tribe resulting from $\tribe$ by deleting the internal vertices of the copy of $\hat R$ in $R'$ from each member $R'$ of $\tribe$.
Let $D'$ be the union of all members of $\tribe'$.
Further, let $X$ be the set of pairs $(x,y) \in V(D) \times V(D)$ for which there is a member $R'$ of $\tribe$ such that $x R' y$ is the copy of $\hat R$ in $R'$.
This means that any member of $\tribe'$ is rooted in $X$.

Now we apply \cref{theo:ubiquity_with_starting_set2} to $D'$, $\tribe'$ and $X$, which yields an infinite set $\mathcal{U}$ of disjoint copies of $U$ in $D'$ that are rooted in $X$.
We join any $U \in \mathcal{U}$ with $x R' y$, where $U$ is rooted in $(x,y) \in X$ and $R'$ is a member of $\tribe$ such that $x R' y$ is the copy of $\hat R$ in $R'$.
Since $\tribe$ is forked at $\hat R$, this gives an infinite family of disjoint copies of $R$.
\end{proof}

\section{Conclusion}
Finally, we show how to deduce the main \namecref{theo:dray_main_theorem} from the above results.

\DoubleRayUbiquity*

\begin{proof}
The ``if'' direction is immediate from \cref{theo:double_rays_positive}.
For the ``only if'' direction let $R$ be a double ray with at least one turn, but not an odd number of turns.
If $R$ has an even number of turns, then we are done by \cref{theo:dray_even_number_of_turns}.
Otherwise, $R$ has infinitely many turns and $R$ is either non-periodic or periodic.
In the former case $R$ is non-ubiquitous by \cref{theo:dray_negative_non_periodic}.
In the latter case $R$ is non-ubiquitous by \cref{theo:dray_negative_periodic}.
\end{proof}

\medskip

\bibliography{ref}

\end{document}